\documentclass[12pt]{article}

\usepackage{amsmath,amsthm,amssymb}
\newtheorem{thm}{Theorem}
\newtheorem{prop}[thm]{Proposition}

\theoremstyle{definition}

\newcommand{\Zed}{\mathbb{Z}}

\begin{document}
\title{Obstructing finite surgery}
\author{Margaret I. Doig}
\maketitle

\abstract{For a fixed $p$, there are only finitely many elliptic 3-manifolds given by $p/q$-surgery on a knot in $S^3$. We prove this result by using the Heegaard Floer correction terms (\emph{d}-invariants) to obstruct elliptic manifolds from arising as knot surgery.}


\section{Introduction}
Consider $p/q$-framed Dehn surgery on a knot $K$ in $S^3$, written $S^3_{p/q}(K)$. Two fundamental questions immediately arise: Given a manifold $Y$, which knots $K$ (if any) have $Y$ as a surgery? Given a knot $K$, which manifolds $Y$ arise as surgery on $K$? 

We investigate in detail the knot surgeries giving elliptic manifolds (henceforth called \emph{finite} surgeries):
\begin{thm}\label{thm1}
Fix $p$ an integer. There are only finitely many elliptic $Y$ such that \[Y=S^3_{p/q}(K).\]
\end{thm}
We focus on the dihedral (prism) manifolds, because, for a fixed $p$, there are only finitely many non-dihedral elliptic manifolds with $H_1(Y)$ of order $p$. We focus on hyperbolic knots because the finite surgeries on non-hyperbolic knots are classified, and only finitely many elliptic spaces arise as surgery on non-hyperbolic knots, for a fixed $p$.

\begin{thm} \label{thm1b}
Fix $m>0$. Consider the dihedral manifolds
\[Y_n=\left(-1;(2,1),(2,1),(n,m)\right).\] 
If $Y_n$ is surgery on a hyperbolic knot in $S^3$, then \[-16m<n<16m.\] 
\end{thm}

Here, $Y=\left(b;(b_1,a_1),(b_2,a_2),(b_3,a_3)\right)$ means $Y$ is constructed by taking an $S^1$ bundle $\xi$ over $S^2$ with $c_1(\xi)=b$ and performing Dehn surgery on three regular fibers with framings $-b_1/a_1, -b_2/a_2, -b_3/a_3$.

\paragraph{Background.} Recall that Thurston showed knots are torus knots (whose complements are Seifert fibered), satellite knots (whose complements contain an essential torus), or hyperbolic knots (whose complements admit a metric of constant negative curvature)~\cite{thurston3d}. On the other hand, as a consequence of Perelman's work, prime manifolds which are surgery on a knot in $S^3$ are small Seifert fibered, toroidal, or hyperbolic~\cite{perelman}. 

We know that surgery on the unknot gives rise to $S^3$, $S^1\times S^2$, and lens spaces; in fact, $S^3$ and $S^1\times S^2$ are not given by surgery on any non-trivial knots in $S^3$ (due to Gordon and Luecke~\cite{gordonlueckecomplements} and Gabai~\cite{gabaifoliations}, respectively). 

Moser classified surgery on torus knots, showing that all such surgeries give rise to lens spaces (or sums of lens spaces) or other small Seifert fibered spaces~\cite{moser}. The Berge Conjecture hypothesizes that the knot surgeries which give lens spaces (also known as \emph{cyclic} surgeries) are exactly the surgeries on \emph{primitive/primitive} knots~\cite{bergeconj}. By Culler, Gordon, Luecke, and Shalen~\cite{culetal}, cyclic surgeries on satellite and hyperbolic knots must be integral. Using Heegaard Floer theory, Ozsv\'ath and Szab\'o verified that Berge's list includes all lens spaces which are knot surgery for $p\leq 1500$~\cite{ozszlens} 
and Greene for all $p$~\cite{greeneberge}. These methods do not identify the knots which may give rise to these surgeries, however. 

Dean attempted to extend these results from lens spaces to small Seifert fibered spaces, hypothesizing that these manifolds are exactly the surgeries on the \emph{primitive/Seifert fibered} knots~\cite{dean}. However, his list is not exhaustive: Deruelle, Mattman, Miyazaki, and Motegi showed that other  surgeries on hyperbolic knots can also produce small Seifert fibered manifolds (although not elliptic ones), but all such known manifolds are also given by knots from Dean's list~\cite{dmm, mmm}. 

Surgery on satellite knots can give rise to small Seifert fibered manifolds (including some elliptic manifolds), but Gordon showed that any atoroidal manifold which is surgery on a satellite knot is also surgery on a torus knot or a hyperbolic knot~\cite{gordonsatellite}. The finite surgeries on iterated torus knots are explicitly listed by Bleiler and Hodgson, and all of the resulting manifolds are also torus knot surgeries~\cite{bleilerhodgson2}; Boyer and Zhang proved that no other satellite knots have finite surgeries~\cite{boyerzhang}. 

The case of small Seifert fibered manifolds which are surgery on hyperbolic knots is more complicated. Boyer and Zhang showed that finite surgery on a hyperbolic knot is either integral or half-integral, and it is conjectured that it is integral (see Problem 177, Conjecture A,~\cite{kirby}). Fintushel and Stern noted that $17-$surgery on the $(-2,3,7)$ pretzel knot is finite~\cite{fintushelsternlens}, and Bleiler and Hodgson commented that $18-$ and $19-$surgery on the $(-2,3,7)$ pretzel knot are cyclic, and $22-$ and $23-$surgery on the $(-2,3,9)$ pretzel knot are finite (although all the resulting manifolds are also torus knot surgeries)~\cite{bleilerhodgson2}. Mattman and then Futer, Ishikawa, Kabaya, Mattman, and Shimokawa showed that there are no other cyclic or finite surgeries on pretzel knots ~\cite{mattmanetalpretzel, mattmanpretzelfinite}. Eudave-Mu\~{n}oz listed some additional hyperbolic knots which may have finite surgeries \cite{eudavemunoz}. Additionally, any hyperbolic knot has at most five finite or cyclic surgeries, with at most one of the surgeries non-integral. Any two such surgeries have distance\footnote{A surgery coefficient $p/q$ corresponds to a homology class $p\mu+q\lambda$ on $\partial \overline{N(K)}$. The \emph{distance} $\Delta(p_1/q_1, p_2/q_2)$ is the minimum geometric intersection number of two curves representing the corresponding homology classes.} at most 3, and the distance 3 is realized by at most one pair~\cite{boyerzhangii}.  

Ghiggini showed that the Poincar\'e homology sphere has a unique surgery description~\cite{ghigginigenusonefibered}. The other finite non-cyclic surgeries with $|H_1(K)|\leq 4$ are also unique, up to orientation~\cite{doigfinite}. Many dihedral manifolds cannot be realized as any knot surgeries, and, in particular, finite surgeries on hyperbolic knots must have surgery coefficient \mbox{$p\geq 10$.} We extend this result in Theorem~\ref{thm1}: given a family of dihedral manifolds with $H_1(Y)$ of a fixed order, at most finitely many arise as surgery on a knot in $S^3$.

\paragraph{Outline.} In Section~\ref{sectsfssurgery}, we list the manifolds with finite, non-cyclic fundamental group. By a theorem of Seifert (Theorem~\ref{thmseifert} below), they are the icosahedral, octahedral, tetrahedral, and dihedral manifolds. The first three cases are knot surgery by Proposition~\ref{thmtorus}. The last case, dihedral manifolds, are the $Y_n=(-1; (2,1),(2,1),(n,m))$ with $H_1(Y)$ of order $4m$. In Section~\ref{sectdinv}, we calculate the Heegaard Floer correction terms (also called $d$-invariants) of the manifolds $Y_n$ in Theorem~\ref{thmdihedrald} and note that the correction terms of the $Y_n$ are unbounded:
\begin{thm}\label{lemmadunbounded}
Fix an integer $m>0$. Consider the family of dihedral manifolds
\[Y_n=\left(-1;(2,1),(2,1),(n,m)\right).\] 
where $n\in \mathbb{Z}$. There is a sequence $\sigma_n\in Spin^c(Y_n)$ such that 
\[\lim_{n\rightarrow \infty} \big| d(Y_n,\sigma_n) \big| =\infty.\]
\end{thm}
However, if $Y_n$ is a surgery, its correction terms are bounded by a linear function of $m$. In Section~\ref{sectdinvsurg}, we recall properties of $L$-spaces, including the elliptic manifolds, and when they may arise as surgeries. We place a bound on the correction terms of any $p/q$-surgery on a hyperbolic knot that gives an elliptic manifold:
\begin{thm}\label{lemmadbounded}
Given $m>0$, for any $K$ so that $S^3_{4m}(K)$ is an $L$-space, and for all $\sigma \in Spin^c(S^3_{4m}(K))$,
\[-4m+\frac{4}{7} \leq d(S^3_{4m}(K);\sigma) \leq m-\frac{1}{4}.\]
\end{thm}
We close with the proof of Theorem~\ref{thm1} in Section~\ref{sectproof} and a list of questions for future work in Section~\ref{sectques}.


\section{Seifert fibered spaces and surgery}\label{sectsfssurgery}

By Perelman \cite{perelman}, all manifolds with finite fundamental group are Seifert fibered. Seifert classified the Seifert fibered spaces of finite fundamental group:

\begin{thm}[Seifert \cite{seifert}]\label{thmseifert} The closed, oriented Seifert fibered spaces with finite but non-cyclic fundamental group are exactly those manifolds with base orbifold $S^2$ and the following presentations:
\begin{enumerate}
\item Type I, icosahedral: $(b; (2,a_1),(3,a_2),(5,a_3))$ with $H_1(Y)=\Zed_m$ and $(m,30)=1$.
\item Type O, octahedral:  $(b; (2,a_1),(3,a_2),(4,a_3))$ with $H_1(Y)=\Zed_{2m}$ and $(m,6)=1$.
\item Type T, tetrahedral: $(b; (2,a_1),(3,a_2),(3,a_3))$ with $H_1(Y)=\Zed_{3m}$ and $(m,2)=1$.
\item Type D, dihedral: $(b;\ (2,a_1),(2,a_2),(b_3,a_3))$ with $H_1(Y)=\Zed_{4m}$ and $(m,b_3)=1$ (if $b_3$ is odd) or $H_1(Y) = \Zed_2\times\Zed_{2m}$ with $(m,2b_3)=1$ (if $b_3$ is even).
\end{enumerate}
where $|H_1(Y)| = b_1b_2b_3\left(-1+\frac{a_1}{b_1} + \frac{a_2}{b_2} + \frac{a_3}{b_3}\right)$ and $(a_i, b_i) = 1$. Any integer $m$ meeting the constraints listed for one of the four types corresponds (up to orientation) to a unique Seifert fibered space of types I, O, or T, or a unique infinite family of type D indexed by the integer $b_3$.
 Any choice of $b, a_i,$ and $b_i$ meeting the appropriate relative primality conditions gives a Seifert fibered space, but a canonical choice of presentation is obtained if $b=-1$ and $a_1=a_2=1$. \label{thmh1order}\end{thm} 
 
The $(b_1, b_2, b_3)$ may be called the \emph{multiplicities} of the space. Moser classified torus knot surgeries in 1971 \cite{moser} and showed, in particular, that all manifolds of types I, O, and T are knot surgeries:

\begin{prop} \label{thmtorus}
If $p/q \neq rs$, then $S^3_{p/q}(T_{r,s})$ has multiplicity \mbox{$(r, s, |rsq-p|)$} \cite[Proposition 3.1]{moser}. In particular, \[S^3_{p/q}(T_{r,s}) = \left(-1; (2,1), (3,1), (6q-p,q)\right). \]

Therefore, every manifold of type I, O, or T is surgery on a torus knot. A manifold of type D is surgery on a torus knot precisely when $b_3$ divides $\frac{|H_1(Y)|}{2}\pm1$.
\end{prop}

\begin{proof}
The description of trefoil surgery may be obtained by a straightforward application of Kirby calculus.

Up to orientation, any manifold of type I, O, or T may be written $(-1; (2,1), (3,1), (b_3,a_3))$, which is $\frac{6a_3-b_3}{a_3}$-surgery on $T_{3,2}$. 

A manifold of type D with multiplicities $(2,2,b_3)$ can only be a torus knot surgery if the knot is $T_{r,2}$ and $2rq-p=\pm 2$. (Note that $q/(2rq-p)$ is a reduced fraction since $(p,q)=1$.) Then $p=|H_1(Y)|$ and $q=(|H_1(Y)|\pm2)/2b_3$, i.e., $b_3$ divides $\frac{|H_1(Y)|}{2}\pm1$.
\end{proof}


\section{The correction terms of $Y_n$}\label{sectdinv}

First, we establish the correction terms of a dihedral manifold.

\begin{thm}\label{thmdihedrald}
Fix an integer $m>0$. Consider the family 
\[Y_n=\left(-1;(2,1),(2,1),(n,m)\right)\] 
with $n>0$
. There is an ordering of $Spin^c(Y_n)=\{\sigma_n^0, \sigma_n^1, \cdots, \sigma_n^{4m-1}\}$ so that
\[d(Y_{n+m},\sigma_{n+m}^i)-d(Y_n,\sigma_n^i)= \left\{\begin{array}{ll}
 -\frac{1}{4} & \mathrm{if}~0\leq i<2m\\
\phantom{-}0 & \mathrm{if}~2m\leq i<4m
\end{array}\right.\]
Similarly, there is an ordering on $Spin^c(Y_{-n})=\{\sigma_{-n}^0, \sigma_{-n}^1, \cdots, \sigma_{-n}^{4m-1}\}$ so that
\[d(Y_{-{n+m}},\sigma_{-n}^i)-d(Y_{-n},\sigma_{-n}^i) =
\left\{\begin{array}{ll}
\frac{1}{4} & \mathrm{if}~0\leq i<2m\\
0 & \mathrm{if}~2m\leq i<4m
\end{array}\right.\]
\end{thm}

\begin{proof}
Assume $n>0$ and $(m,n)=1$. The family $\{Y_n\}$ constitutes all the Seifert fibered manifolds with $|H_1(Y)|=4m$ up to orientation; if $-n$ is negative, $d(Y_{-n},-\sigma)=-d(Y_n,\sigma)$ for some reasonable matching of $Spin^c(Y_{-n})$ to $Spin^c(Y_n)$. We actually consider the family of manifolds \[-Y_n=\left(-2;(2,1),(2,1),(n,n-m)\right)\] since the corresponding plumbing diagrams have negative definite intersection forms. 

N\'emethi \cite[Section 11.13]{nemethiozszinvariants} outlines one method for calculating $d(Y,\sigma)$ for Seifert fibered spaces with negative definite plumbing diagrams:
\[d(Y,\sigma)=\frac{K^2+s}{4}-2\chi(l')-2\min_{i\geq 0} \tau(i)\] 

\paragraph{1.} Let $-Y_n$ be constructed by taking a line bundle of a circle and then performing three Dehn surgeries with coefficients $-\alpha_i/\omega_i$. Define $e_0$ to be the first Chern class of the line bundle. Let $\omega_i'$ be between $0$ and $\alpha_i$ so that $\omega_i\omega_i'=1\bmod \alpha_i$. That is,
\[\begin{array}{lll}
e_0=-2\\
\alpha_1=2 & \omega_1=1 & \omega_1'=1 \\
\alpha_2=2 & \omega_2=1 & \omega_2'=1 \\
\alpha_3=n & \omega_3=n-m & \omega_3' \equiv -1/m \bmod n
\end{array}\]
Consider also the invariants $e=e_0+\sum_{l=1}^3\frac{\omega_l}{\alpha_l}$, and $\varepsilon=(2-3+\sum_{l=1}^3\frac{1}{\alpha_l})/e$.
Then
\[e=-\frac{m}{n} \qquad \varepsilon=-\frac{1}{m}\]
Similarly, substituting $n+m$ for $n$, it is possible to see that $-Y_{n+m}$ has the invariants
\[\begin{array}{lll}
e_0=-2 & e=-\frac{m}{n+m} & \varepsilon=-\frac{1}{m}\\
\alpha_1=2 & \omega_1=1 & \omega_1'=1 \\
\alpha_2=2 & \omega_2=1 & \omega_2'=1 \\
\alpha_3=n+m & \omega_3=n & \omega_3' \equiv -1/m \bmod n+m
\end{array}\]

\paragraph{2.} $K^2+s$ is defined
\[K^2+s = \varepsilon^2e+e+5-12\sum_{l=1}^3 \textbf{s}(\omega_l,\alpha_l)\]
where $\textbf{s}(\omega_l,\alpha_l)$ is a Dedekind sum. Therefore, by Proposition~\ref{propdedekind} below, $K^2+s$ differs for $-Y_{n+m}$ and $-Y_n$ by 
\begin{multline*}(K^2+s)_{n+m}-(K^2+s)_n\\
 = -\frac{m}{n}-\frac{1}{mn}+\frac{m}{n-m}+\frac{1}{m(n-m)}
 -12\big(\textbf{s}(n-m,n)-\textbf{s}(n-2m,n-m)\big)\end{multline*}
 and, noting that $\textbf{s}(n-2m,n-m)=\textbf{s}(-m,n-m)$ and applying Proposition~\ref{propdedekind},
\[
  = -\frac{m}{n}-\frac{1}{mn}+\frac{m}{n-m}+\frac{1}{m(n-m)} +3-\frac{n}{n-m}-\frac{n-m}{n}-\frac{1}{n(n-m)}=1.\]

\paragraph{3.}Each $\sigma_n^i\in Spin^c(-Y_n)$ corresponds to one of the $4m$ distinct integer vectors $(a_0,a_1,a_2,a_3)$ which satisfy the equations
\begin{equation}\label{eqnSIred}\left\{\begin{array}{ll}
0 \leq a_0; ~ 0 \leq a_l < \alpha_l & l =1,2, 3\\
s(i)=1+a_0+ie_0 + \displaystyle\sum_{l=1}^3\left \lfloor\frac{i\omega_l+a_l}{\alpha_l}\right\rfloor \leq 0  &\forall i>0
\end{array}\right.\end{equation}
In this case,
\[\left\{\begin{array}{ll}
0 \leq a_0; \quad 0 \leq a_1 <2; \quad 0\leq a_2<2; \quad 0\leq a_3<n\\
s(i)=1+a_0-2i + \Big \lfloor\frac{i+a_1}{2}\Big\rfloor +  \Big \lfloor\frac{i+a_2}{2}\Big\rfloor +  \left \lfloor\frac{i(n-m)+a_3}{n}\right\rfloor \leq 0  &\forall i>0
\end{array}\right.\]


N\'emethi's Theorem~11.5 says that these equations have exactly $|H_1(Y_n)|=4m$ integral solutions. If $(a_0,a_1,a_2,a_3)$ is a solution and $a_0>0$, then $(a_0-1,a_1,a_2,a_3)$ is a solution; if $a_1>0$ also, then $(a_0,a_1-1,a_2,a_3)$ is a solution; etc. 

Check that $(0,0,0,2m-1)$ satisfies $s(i)\leq 0$ if $i>0$: observe that $s(1)= -1, s(2)=0$, and
\begin{multline*}s(i+2)-s(i)\\
=1-2(i+2)+2\left\lfloor\frac{i+2}{2}\right\rfloor+\left\lfloor\frac{(i+2)(n-m)+a_3}{n}\right\rfloor\\
-1+2i-2\left\lfloor\frac{i}{2}\right\rfloor-\left\lfloor\frac{i(n-m)+a_3}{n}\right\rfloor \\ 
\leq -2+\left\lfloor\frac{2(n-m)}{n}\right\rfloor\leq 0.\end{multline*}
However, $(0,0,0,2m)$ is not: $s(2)=1$. 

Similarly, $(0,0,1,m-1)$  and $(0,1,0,m-1)$ are solutions, as $s(1)=0$ and $s(i+1)-s(i)\leq 0$. On the other hand, $(0,0,1,m)$ and $(0,1,0,m)$ have $s(1)=1$. 

Finally, note that $(1,0,0,0)$ is not a solution: $s(2)=\left\lfloor\frac{2(n-m)}{n}\right\rfloor=1$ if $n>2m$.

Therefore, if $n>2m$, the integral solutions are
\begin{equation}\label{eqnvectors}\begin{split}
(0,0,0,a_3)  \qquad \mathrm{for}~a_3=0,1,\cdots, 2m-1\\
(0,0,1,a_3) \qquad \mathrm{for}~a_3=0,1,\cdots, m-1\phantom{2}\\
(0,1,0,a_3) \qquad \mathrm{for}~a_3=0,1,\cdots, m-1\phantom{2}
\end{split}\end{equation}
In particular, the equations for $-Y_n$ and $-Y_{n+m}$ have the same solutions.

\paragraph{4.} Next, each $(a_0,a_1,a_2,a_3)$ induces a
\[-\chi=\sum_{l=0}^3\frac{a_l}{2} +\frac{\varepsilon\tilde{a}}{2}+\frac{\tilde{a}^2}{2e}-\sum_{l=1}^3\sum_{i=1}^{a_l}\left\{\frac{i\omega'_l}{\alpha_l}\right\}\]
where $\tilde{a}=a_0+\sum_{l=1}^3\frac{a_l}{\alpha_l}$ and $\{x\}$ is the fractional part of $x$. 

Consider $(0,0,0,a_3)$, which is a solution for (\ref{eqnSIred}) for both $-Y_{n+m}$ and $-Y_n$. First,
\[\big(\tilde{a}\big)_{n}=\frac{a_3}{n} \qquad (\tilde{a})_{n+m}=\frac{a_3}{n+m}\]
Next, compare $-\chi$ for $-Y_{n+m}$ and $-Y_n$. First recall $\omega_3'$ is chosen so $\omega_3'\omega_3\equiv 1 \bmod \alpha_3$. Considering the invariants for $-Y_n$, we see that $(\omega_3')_n(n-m)\equiv 1\pmod n$, or, equivalently, $(\omega_3')_n(-m)\equiv 1\pmod n$, which means 
\[\left\{\frac{(\omega_3')_n}{n}\right\}=\left\{-\frac{1}{mn}\right\}.\] 
Therefore,
\begin{multline*}-2\big((\chi )_{n+m}-(-\chi)_{n}\big)\\
=\frac{a_3^2+a_3}{n(n+m)}-2\sum_{i=1}^{a_3}\left\{\frac{i(\omega'_3)_{n+m}}{n+m}\right\}+2\sum_{i=1}^{a_3}\left\{\frac{i(\omega'_3)_n}{n}\right\}\\
=\frac{a_3^2+a_3}{n(n+m)}-2\sum_{i=1}^{a_3}\left\{\frac{-i}{(n+m)m}\right\}+2\sum_{i=1}^{a_3}\left\{\frac{-i}{nm}\right\}\end{multline*}
We can combine the fractional parts (since $\{a\}+\{b\}=\{a+b\}$ when none of the three is an integer), and then remove the fractional signs (since $0 < i < n(n+m)$) and simplify:
\[
=\frac{a_3^2+a_3}{n(n+m)}-2\sum_{i=1}^{a_3}\left\{\frac{i}{n(n+m)}\right\} = 0\]


By a similar method, the solution $(0,1,0,a_3)$ gives
\[\big(\tilde{a}\big)_{n}=\frac{a_3}{n} +\frac{1}{2} \qquad (\tilde{a})_{n+m}=\frac{a_3}{n+m}+\frac{1}{2}\]
so that 
\[-2\big((\chi )_{n+m}-(-\chi)_{n}\big) =\frac{a_3^2+a_3}{n(n+m)}-\frac{1}{4}-2\sum_{i=1}^{a_3}\frac{i}{n(n+m)}=-\frac{1}{4}\]

\paragraph{5.} We prove $\min_{i\geq 0} \tau(i)= \tau(0)=0$. The $\tau(i)$ are defined by setting $\tau(0)=0$ and 
\[\tau(i+1)-\tau(i)=1+a_0-ie_0+\sum_{l=1}^3\left\lfloor\frac{-i\omega_l+a_l}{\alpha_l}\right\rfloor\]
when $i>0$. For $(a_0,a_1,a_2,a_3)=(0,0,0,a_3)$ where $0\leq a_3<2m$, 
\begin{multline*}\tau(i+1)-\tau(i)=1+2i+2\left\lfloor -\frac{i}{2}\right\rfloor + \left\lfloor \frac{-i(n-m)+a_3}{n}\right\rfloor\\
 \geq i+\left\lfloor-\frac{i(n-m)}{n}\right\rfloor+\left\lfloor\frac{a_3}{n}\right\rfloor\geq \left\lfloor\frac{a_3}{n}\right\rfloor=0 \quad \forall i\geq 0
\end{multline*}
For $(0,1,0,a_3)$ or $(0,0,1,a_3)$ with $0\leq a_3<m$,
\begin{multline*}\tau(i+1)-\tau(i)=1+2i+\left\lfloor -\frac{i}{2}\right\rfloor + \left\lfloor -\frac{i+1}{2}\right\rfloor +\left\lfloor \frac{-i(n-m)+a_3}{n}\right\rfloor\\
 \geq i+\left\lfloor-\frac{i(n-m)}{n}\right\rfloor+\left\lfloor\frac{a_3}{n}\right\rfloor\geq \left\lfloor\frac{a_3}{n}\right\rfloor=0 \quad \forall i\geq 0
\end{multline*}
which means $\tau(i)$ is increasing and
\[\min_{i\geq 0} \tau(i)= \tau(0)=0.\]

\paragraph{6.} Finally, we calculate 
\[d(-Y_{n+m},\sigma_{n+m})= d(-Y_n,\sigma_n) + \frac{1}{4}\]
for $\sigma_{n+m}$ and $\sigma_n$ the $Spin^c$ structures that correspond to $(0,0,0,a_3)$, and
\[d(-Y_n,\sigma_{n+m})= d(-Y_n,\sigma_n)\]
for $\sigma_{n+m}$ and $\sigma_n$ corresponding to $(0,0,1,a_3)$ or $(0,1,0,a_3)$. Reversing orientation and making a reasonable choice of ordering on $Spin^c(Y_n)$ gives the theorem statement.
\end{proof}

\begin{prop}\label{propdedekind}
If $m>0$ and $\gcd(2m,n)=1$,
\[\textbf{s}(n,n-m)+\textbf{s}(-m,n-m)=0\]
and the Dedekind sum reciprocity formula says:
\[\textbf{s}(a,b)+\textbf{s}(b,a)=\frac{1}{12}\left(\frac{a}{b}+\frac{b}{a}+\frac{1}{ab}\right)-\frac{1}{4}
\]
\end{prop}
\begin{proof}

\[\textbf{s}(p,q)=\displaystyle\sum_{i=1}^{q-1}\left(\left(\frac{i}{q}\right)\right)\left(\left(\frac{pi}{q}\right)\right)\]
where
\[\left(\left(x\right)\right) = \left\{\begin{array}{ll} x-\lfloor x \rfloor - \frac{1}{2} &   x\notin \mathbb{Z}\\0 & x\in \mathbb{Z}\end{array}\right.\]
so
\begin{multline*}\textbf{s}(n,n-m)+\textbf{s}(-m,n-m)\\
=\displaystyle\sum_{i=1}^{n-m-1}\left(\left(\frac{i}{n-m}\right)\right)\left(\left(\left(\frac{ni}{n-m}\right)\right)+\left(\left(-\frac{mi}{n-m}\right)\right)\right)\\
=\displaystyle\sum_{i=1}^{n-m-1}\left(\left(\frac{i}{n-m}\right)\right)\left(i-\left\lfloor\frac{ni}{n-m}\right\rfloor+\left\lfloor-\frac{mi}{n-m}\right\rfloor -1\right)=0\end{multline*}
since $\left\lfloor\frac{ni}{n-m}\right\rfloor+\left\lfloor-\frac{mi}{n-m}\right\rfloor =i-1$ as long as $(n-m) \nmid i$.

The second equality is the Dedekind sum reciprocity formula (see~\cite[Chapter~2]{dedekindsum}).
\end{proof}

\section{The correction terms of $S^3_{4m}(K)$}\label{sectdinvsurg}
Next, we analyze the correction terms of integral $L$-space surgeries.



\begin{proof}[Proof of Theorem~\ref{lemmadbounded}]
Say $4m>0$ and $Y=S^3_{4m}(K)$ is an $L$-space, the Heegaard Floer analogue of a lens space. That is, $\widehat{HF}(Y,\sigma)\cong\widehat{HF}(S^3)$. Ozsv\'ath-Szab\'o developed a performed a careful analysis of knot surgeries, including an analysis of how the gradings of $\widehat{HF}(Y)$ compare to the gradings of $\widehat{HF}(Y_n(K))$ \cite{ozszzsurgery,ozszqsurgery}. We apply their results to this situation. 

By \cite[Corollary 1.4]{ozszqsurgery}, if $K$ admits a positive $L$-space surgery, then $S^3_{p/q}(K)$ is an $L$-space iff $\frac{p}{q} \geq 2 g(K) -1.$ For $p/q=4m$, \[g(K)\leq 2m.\]

Use the identification $Spin^c(Y) \cong \Zed_{4m}$ from \cite[Proposition~4.8]{ozszabsolute}. Then $d(Y,i)=d(Y,4m-i)$, so consider only \mbox{$0\leq i \leq p/2=2m$.} 

Next, $d(S^3_{p/q}(K),i) = d(S^3_{p/q}(U),i) -2 \sum_{j=1}^{\infty} ja_{|\lfloor i/q\rfloor|+j}$ \cite[Theorem 1.2]{ozszqsurgery}, so
\[d(S^3_{4m}(K),i)=d(S^3_{4m}(U),i)-2 \sum_{j=1}^{\infty} ja_{i+j}.\]
Also, $d(S^3_{p/q}(U),i) = -\left( \frac{p q-(2i+1-p-q)^2}{4p q}\right) - d(S^3_{q/r}(U),j)$
where $r \equiv p \bmod q$ and $j \equiv i \bmod q$ \cite[Proposition 4.8]{ozszabsolute}. Therefore,
\[\displaystyle d(S^3_{4m}(U),i)=-\frac{4m-(2i-4m)^2}{16m} = -\frac{1}{4}+\frac{(i-2m)^2}{4m},\]
but we only consider $0\leq i \leq 2m$. By the second derivative test, the minimum occurs at $i=2m$ and the maximum at $i=0$:
\[\displaystyle -\frac{1}{4} \leq d(L_{4m,1},i) \leq m-\frac{1}{4}\]
By Proposition~\ref{propalex} below, as long as $|i|<g(K)$,
\[0 \leq \sum_{j=1}^{\infty} ja_{|i|+j} \leq g(K)-1\leq 2m-1\] 
Finally, 
\[\displaystyle -4m+\frac{7}{4} \leq d(S^3_{4m}(K),i)\leq m-\frac{1}{4}\]
\end{proof}

\begin{prop} \label{propalex}
If $K$ is a knot with an $L$-space surgery with genus $g=g(K)>1$ and Alexander polynomial coefficients $\{a_j\}$, and if $|i| < g(K)$, then 
\begin{equation}0 \leq \sum_{j=1}^{\infty} ja_{|i|+j} \leq\max(1, g(K)-|i|-1)\label{eqnalex}\end{equation}
\end{prop}
\begin{proof}
Recall that the Alexander polynomial may be normalized to \[\Delta_K(T)=a_0+\sum_{j=1}^g a_j(T^j+T^{-j}).\] If $K$ has $L$-space surgeries, then the non-zero $a_i$ are alternating $+1$s and $-1$s where the highest non-zero term is $a_g=+1$ \cite{ozszlens}. Write the non-zero coefficients $\{a_{j_l}\}_{l=1}^k$. Note $k\geq 1$.

If $k$ is even, $a_{j_{2l}}=+1$ and $a_{j_{2l-1}}=-1$, so
\[\sum_{j=1}^{g} ja_j=\sum_{l=1}^{k} j_la_{j_l} = \sum_{j=1}^{k/2}(j_{2l}-j_{2l-1}) \geq 0.\]
If $k=1$, 
\[\sum_{j=1}^{g} ja_j= ga_g \geq 0.\]
If $k>1$ is odd, $a_{j_{2l}}=-1$ and $a_{j_{2l-1}}=+1$, and
\[\sum_{j=1}^{g} ja_j=\sum_{l=1}^{k} j_la_{j_l} = j_0 + \sum_{l=1}^{(k-1)/2}(j_{2l+1}-j_{2l}) \geq 0.\]
Similarly, if the top term $a_g$ had been negative instead of positive, then $\sum_{j=1}^{g} ja_j \leq 0$. Therefore,
\[\sum_{j=1}^{g} ja_j =ga_g+\sum_{j=1}^{g-1} ja_j\leq g.\]

If $K$ has $L$-space surgeries, it is actually known that the second highest term is $a_{g-1}=-1$ \cite{mysterymatt}, so, since $g(K)>1$,
\[ \sum_{j=1}^{g} ja_{j} = g-(g-1) + \sum_{j=1}^{g-2}ja_j \leq g-1\]
and, as long as $|i|<g-1$,
\[ \sum_{j=1}^{g} ja_{|i|+j} =\sum_{j=1}^{g-|i|} ja_{|i|+j} \leq g-|i|-1\]
and, if $|i|=g-1$,
\[ \sum_{j=1}^{g} ja_{|i|+j} =1\]

\end{proof}

\section{Proof}\label{sectproof}

We finish with a proof of the main result.

\begin{proof}[Proof of Theorem~\ref{thm1b}]
If $Y$ is one of the infinite family of dihedral manifolds $Y_n$, by Theorem~\ref{lemmadunbounded}, there is a sequence of $\sigma_n\in Spin^c(Y_n)$ such that 
\[\lim_{n\rightarrow \infty} \big| d(Y_n,\sigma_n) \big| =\infty.\]
However, by Theorem~\ref{lemmadbounded},
\[-4m \leq d(S^3_{4m}(K);\sigma) \leq m \qquad \forall \sigma \in Spin^c(S^3_{4m}(K))\]
Therefore, for sufficiently large $n$, $Y_n$ is not surgery on a hyperbolic knot.
\end{proof}

\begin{proof}[Proof of Theorem~\ref{thm1}]
Finite surgery on non-hyperbolic knots is classified (see Section~\ref{sectsfssurgery}). In particular, for a fixed $p$, there are finitely many choices of $K$ and $p/q$ so that $S^3_{p/q}(K)$ is elliptic \cite{bleilerhodgson2,boyerzhang,moser}.

Fix $Y$ with finite fundamental group and $H_1(Y)=\mathbb{Z}_p$. By Perelman \cite{perelman}, $Y$ is Seifert fibered. If it is surgery on a hyperbolic knot, then $q=1$ by \cite[Theorems~1.1,~1.2]{boyerzhangii}. By Seifert (Theorem~\ref{thmseifert}), if $p$ is not divisible by 4, there is at most one such $Y$, and it is icosahedral, octahedral, or tetrahedral; for each of these, it is knot surgery by Proposition~\ref{thmtorus}. If $p$ is divisible by 4, $Y$ is one of the infinite family of dihedral manifolds $Y_n$. 
\end{proof}

\section{Questions}\label{sectques}

What is the best bound on $n$ that ensures $Y_n$ is not surgery on a knot? The bounds in Theorem~\ref{lemmadbounded} hold for $|n|\geq 16m$, but the largest $n$ for which the author knows $Y_{\pm n}$ is a surgery is $n=2m+1$ (see Proposition~\ref{thmtorus} with $T_{2m+1,2}$. Is this bound sharp? If not, can it be improved by applying additional information about $K$ to narrow down the possible correction terms?

Do any of Dean's knots give $Y_n$ for $|n| > 2m+1$?

Of the remaining manifolds with $|n|\leq 16m$, which are knot surgeries? How good are the correction terms at obstructing these cases from being surgery on a knot? In \cite{doigfinite}, a more painstaking examination of the correction terms for all $|n|\leq 32$ obstructed almost all $Y_n$ from being surgery on a knot.

How can the correction terms be used to study Dean's list and test whether it is comprehensive for finite surgeries?


\section{Acknowledgements}

This work would never have been possible without the constant interaction and mathematical stimulation of all my colleagues from Indiana University, Bloomington; Princeton University; and the many seminars and conferences I have been privileged to attend.

\bibliographystyle{plain}
\bibliography{biblio}
\end{document}